\documentclass[12pt]{amsart}

\usepackage{fullpage}
   \usepackage[backend=bibtex]{biblatex}
\usepackage{amsfonts}
\usepackage{amsmath}
\usepackage{amssymb}
\usepackage{amsthm}
\usepackage{mathrsfs}
\usepackage{mathrsfs}
\usepackage{tikz}
\usepackage[margin=1.1in]{geometry}
\usetikzlibrary{arrows,matrix}

\newcommand{\ZZ}{\mathbb{Z}}
\newcommand{\RR}{\mathbb{R}}

\numberwithin{equation}{section}

\theoremstyle{plain}
\newtheorem{theorem}[equation]{Theorem}
\newtheorem{lemma}[equation]{Lemma}
\newtheorem{proposition}[equation]{Proposition}
\newtheorem{corollary}[equation]{Corollary}

\theoremstyle{definition}
\newtheorem{definition}[equation]{Definition}
\newtheorem{example}[equation]{Example}

\theoremstyle{remark}
\newtheorem*{remark}{Remark}
\newtheorem*{example*}{Example}

 { \begin{list}%
         {$\bullet$}%
         {\setlength{\labelwidth}{20pt}%
          \setlength{\leftmargin}{25pt}%
          \setlength{\topsep}{0pt}
          \setlength{\itemsep}{1ex}
          \setlength{\parsep}{0pt}}}%
 { \end{list} }

\title{Symmetric and $r$-Symmetric Tropical Polynomials and Rational Functions}
\author{Gunnar Carlsson}
\address{Mathematics Department,
         Stanford University University,
         Building 380, Sloan Hall,
         Stanford, California, 94305,
         USA}
\email{Gunnar@math.stanford.edu}

\author{Sara Kali\v{s}nik Verov\v{s}ek}
\address{Mathematics Department,
         Stanford University University,
         Building 380, Sloan Hall,
         Stanford, California, 94305,
         USA}
\email{Kalisnik@math.stanford.edu}
 
\begin{document}
 \maketitle
\begin{abstract}
A tropical polynomial in $nr$ variables, divided into blocks of $r$ variables each, is $r$-symmetric if it is invariant under the action of $S_n$ that permutes the blocks. For $r=1$ we call these symmetric tropical polynomials. We can define $r$-symmetric and symmetric tropical rational functions in a similar manner. In this paper we identify generators for the sets of symmetric tropical polynomials and rational functions. While $r$-symmetric tropical polynomials are not finitely generated for $r \geq 2$, we show that $r$-symmetric tropical rational functions are and provide a list of generators.
\end{abstract}
{\bf Keywords:} Tropical Algebra, Tropical Polynomials, Max-Min Algebra, Persistence Barcode

\section{Introduction}

Tropical geometry has been developed over the last two decades  to understand a wide variety of problems. 
  Tropical polynomial problems are interpretable as linear problems, with inequalities. This makes the methods useful in various kinds of optimization problems.  Tropical geometry can also be used to approximate, in an appropriate sense, ordinary algebraic geometric problems, and therefore is useful in their solution.    In addition, it permits the solution to a number of enumerative geometric problems~\cite{Mikhalkin}.  Although much work has been done, it is clear that there is not yet a complete translation of the methods of algebraic geometry to the tropical situation.  One of the main objects of study in algebraic geometry, invariant theory, has to our knowledge not been studied at all in the tropical setting.  In this paper, we initiate the translation of invariant theory by studying some special cases.  We find that although some ideas from ordinary invariant theory translate, there are interesting differences.  

We call a tropical polynomial in $n$ variables symmetric if it is invariant under the action of $S_n$ that permutes the variables. The first result of the paper (Theorem 3.7) states that every tropical symmetric polynomial can be written as a tropical polynomial in the elementary symmetric tropical polynomials $e_1, \ldots, e_n$ and $e_n^{-1}$. In Corollary 3.11 we show that this expression is unique if the representation of the polynomial is `minimal' (essentially every monomial involved contributes to the function). This is generalized to rational functions in Theorem 4.2.

In the second part of the paper we consider the case when the tropical polynomial semiring has $nr$ variables that come in $n$ blocks of $r$ variables each and are permuted by the symmetric group $S_n$. We call a tropical polynomial in $nr$ variables $r$-symmetric if it is invariant under the action of $S_n$ that permutes the blocks.
We define elementary $r$-symmetric polynomials and show that they separate orbits. As opposed to the ordinary polynomial case, the semiring of $r$-symmetric tropical polynomials is not finitely generated, but the elementary $r$-symmetric polynomials do generate $r$-symmetric rational functions.

Our motivation for the study of this problem came out of the study of persistence barcodes~\cite{topodata},\cite{pattern},\cite{ZC}.  The interpretation of the space of persistence barcodes as embedded in the geometric points of an affine scheme over $\RR$ led us to the conclusion that some important functions were not included, notably max and min functions, which suggested to us that it would be valuable to carry out a parallel analysis where tropical functions on the barcodes are studied.

\section{Tropical Polynomials and Tropical Rational Functions}\label{basics}
Tropical algebra is based on the study of the tropical semiring $(\RR \cup \{\infty \}, \oplus, \odot)$. In this semiring,  addition and multiplication are defined as follows:
\[
\begin{array}{ccc}
a\oplus b := \min{(a, b)} &\, \textrm{and} \,& a\odot b := a+b.
\end{array}
\]
Both are commutative and associative. The times operator $\odot$ takes precedence when plus $\oplus$ and times $\odot$ occur in the same expression.
The distributive law holds:
\[
a \odot (b\oplus c) = a\odot b \oplus a\odot c.
\]
Moreover, the Frobenius identity (Freshman's Dream) holds for all powers $n$ in tropical arithmetic:
\begin{equation}\label{freshman}
(a \oplus b)^n = a^n \oplus b^n.
\end{equation}
Expression $b^{-1}$ is the inverse of $b$ with respect to $\odot$ and equals $-b$ in ordinary arithmetic.

Let $x_1, x_2, \ldots, x_n$ be variables representing elements in the tropical semiring. A \emph{tropical monomial expression} is any product or quotient of these variables, where repetition is allowed. By commutativity, we can sort the product and write monomial expressions with the variables raised to exponents.

\begin{definition}
A \emph{tropical polynomial expression} is a finite linear combination of tropical monomial expressions:
\[
p(x_1, x_2, \ldots, x_n) = a_1\odot x_1^{i_1^1} x_2^{i_2^1} \ldots x_n^{i_n^1} \oplus a_2\odot x_1^{i_1^2} x_2^{i_2^2} \ldots x_n^{i_n^2}\oplus \ldots \oplus a_m\odot x_1^{i_1^m} x_2^{i_2^m} \ldots x_n^{i_n^m}.
\]
Here the coefficients $a_1, a_2, \ldots a_m$ are real numbers and the exponents $i_k^j$ for ${1\leq k \leq n}$ and $1\leq j \leq m$ are integers. 
\end{definition}

The \emph{total degree of an expression} $p(x_1, x_2, \ldots, x_n)$ is
\[
\operatorname{deg} p = {\max_{1\leq j \leq m}(i_1^j+ i_2^j + \ldots + i_n^j)}.
\]
Tropical expressions form a semiring and are sometimes called tropical polynomials in other sources (see \textit{Tropical Mathematics}~\cite{speyersturm} or \textit{Introduction to Tropical Geometry}~\cite{tropintro}).

Each tropical polynomial expression represents a concave piece-wise linear function from ${(\RR\cup\{\infty \})^n}$ to ${\RR \cup \infty}$. Tropical polynomial expressions whose image is contained in $\RR$ are $\RR$-tropical polynomial expressions.

\begin{example}
Let $n=3$. A tropical monomial expression
\[
x_2 \odot x_1 \odot x_3 \odot x_2 \odot x_2 \odot x_1 = x_1^2 \odot x_2^3 \odot x_3 =x_1^2 x_2^3x_3
\]
represents a linear function
\[
(x_1, x_2, x_3) \mapsto x_2 +x_1 +x_3 +x_2 +x_2 +x_1 = 2x_1 +3 x_2 + x_3.
\]
\end{example}

The passage from tropical polynomial expressions to functions is not one-to-one. For example, 
\[
x_1^2 \oplus x_2^2 = x_1^2 \oplus x_2^2 \oplus x_1x_2.
\]
Let  $p(x_1, x_2, \ldots, x_n)$ and $q(x_1, x_2, \ldots, x_n)$ be tropical polynomial expressions. If
\[
p(x_1, x_2, \ldots, x_n) = q(x_1, x_2, \ldots, x_n)
\]
for all $(x_1, x_2, \ldots, x_n)\in (\RR \cup \infty)^n$, then $p$ and $q$ are \emph{functionally equivalent}. Functional equivalence $\sim$ is an equivalence relation on the set of all tropical polynomial expressions.  We are interested primarily in studying functions, so instead of observing the entire semiring of tropical polynomial expressions we identifiy those expressions that define the same functions.

\begin{definition}
Tropical polynomials are the semiring of equivalence classes of tropical polynomial expressions with respect to $\sim$. In the case of $n$ variables we denote it by $\operatorname{Trop}[x_1, x_2, \ldots, x_n]$. 
\end{definition}
\begin{remark}
Note that our tropical polynomial semiring is not obtained using the standard `polynomial semiring' construction. That construction yields the semiring of tropical polynomial expressions. 
\end{remark}

The degree of a tropical polynomial $p$ is
\[
\textrm{Deg} p = \min_{\{q\,|\,q\sim p\}} \operatorname{deg} q,
\]
where the minimum is taken over all tropical polynomial expressions $q$ which represent $p$. 

In the semiring of tropical polynomials $\operatorname{Trop}[x_1, x_2, \ldots, x_n]$ the operation $\odot$ is not invertible. We have more flexibility to manipulate expressions if we allow inverses with respect to $\odot$. 
\begin{definition}
A tropical rational expression $r$ is a quotient
\[
r(x_1, \ldots, x_n) = p(x_1, \ldots, x_n) \odot q(x_1, \ldots, x_n)^{-1},
\]
where $p$ is a tropical polynomial expression and $q$ is an $\RR$-tropical polynomial expression. 
\end{definition}
\begin{remark}
Tropical rational expressions are the localization of the semiring of tropical polynomial expressions with respect to the
multiplicatively closed set of $\RR$-tropical polynomial expressions.
\end{remark}
We say that tropical rational expressions $r$ and $s$ are \emph{functionally equivalent} and write $r\sim s$ if
\[
r(x_1, x_2, \ldots, x_n) = s(x_1, x_2, \ldots, x_n)
\]
for all $(x_1, x_2, \ldots, x_n)\in (\RR \cup \infty)^n$. 

Since
\[
-\operatorname{min}(a, b) = \operatorname{max}(-a, -b),
\]
tropical rational expressions are composed of taking the maxima and minima of linear functions, i.e. the set of tropical rational expressions is the smallest subset of functions $\RR^n\to \RR$ containing all constant maps, projections and closed under $+$, $\min$ and $\max$. 

Conversely, any function from this set can be represented by an expression of the form $p \odot q^{-1}$, where $p$ and $q$ are tropical polynomial expressions. The algorithm to produce $p$ and $q$ is the usual one of adding fractions by finding a common denominator, but performed in tropical arithmetic.

\begin{example}\label{ex}
Let $r(x_1, x_2) = x_1^{-1}\odot x_2 \oplus (x_2)^{-1} \oplus (x_2 \odot x_1 \oplus x_1)^{-1}$. We can write
\[
\begin{array}{lclc}
r(x_1, x_2)&= & \operatorname{min}(-x_1+ x_2, - x_2, -\operatorname{min}( x_2 + x_1,  x_1))&\\
&=& \operatorname{min}(-x_1+ x_2+
\operatorname{min}( x_2 + x_1,  x_1), - x_2+
\operatorname{min}( x_2 + x_1,  x_1), 0) -\operatorname{min}( x_2 + x_1,  x_1)&\\
&=& \operatorname{min}(\operatorname{min}( 2x_2,  x_2), - x_2+
\operatorname{min}( x_2 + x_1,  x_1), 0) -\operatorname{min}( x_2 + x_1,  x_1)&\\
&=&  \operatorname{min}(\operatorname{min}( 3x_2,  2x_2), 
\operatorname{min}( x_2 + x_1,  x_1), x_2) -\operatorname{min}( 2x_2 + x_1,  x_2+ x_1)&\\
&=& \operatorname{min}( 3x_2,  2x_2,  x_1+x_2, x_1, x_2) -\operatorname{min}( 2x_2 + x_1,  x_1+x_2)&\\
&=&(x_2^3 \oplus x_1 \oplus x_2) \odot (x_2^2x_1 \oplus x_1x_2)^{-1}.
\end{array}
\]
\end{example}

\begin{definition}
The semiring of equivalence classes of tropical rational expressions with respect to this relation is $\operatorname{RTrop}[x_1, x_2, \ldots, x_n]$ and is called the \emph{semiring of rational tropical functions}. 
\end{definition}

\section{Symmetric Tropical Polynomials}
\begin{definition}
A tropical polynomial $p \in \operatorname{Trop}[x_1, x_2, \ldots, x_n]$ is symmetric if 
\[
p(x_1,\ldots, x_n) = p(x_{\pi(1)},\ldots ,x_{\pi(n)})
\]
for every permutation $\pi \in S_n$.
\end{definition}
We denote the semiring of symmetric tropical polynomials by $\operatorname{Trop}[x_1, x_2, \ldots, x_n]^{S_n}$. We work with a fixed $n$ throughout this section.

\begin{example}
Let $n=3$.  Tropical polynomials $x_1^2 \oplus x_2^2 \oplus x_3^2$ and $x_1 \odot x_2 \odot  x_3$ are symmetric. 
\end{example}

We define a symmetrization operator:
\[
\begin{array}{llcl}
\operatorname{Sym} \colon& \operatorname{Trop}[x_1, x_2, \ldots, x_n]& \to &\operatorname{Trop}[x_1, x_2, \ldots, x_n]^{S_n} \\
&&&\\
&  p(x_{1}, x_2, \ldots, x_n)& \mapsto & \bigoplus_{\pi\in S_n} p(x_{\pi(1)}, x_{\pi(2)}, \ldots, x_{\pi(n)}).
\end{array}
\]

\begin{proposition}\label{properties}
Let $p, q \in \operatorname{Trop}[x_1, x_2, \ldots, x_n]$ and $a\in \RR$. Then:
\begin{enumerate}
\item
$\operatorname{Sym}(p\oplus q)= \operatorname{Sym}(p) \oplus \operatorname{Sym}(q)$,
\item
$a \odot \operatorname{Sym}(p) = \operatorname{Sym}(a \odot p)$.
\end{enumerate}
\end{proposition}
\begin{proof}
We leave the proof to the reader.
\end{proof}


\begin{proposition}\label{iffsym}
Let $p \in \operatorname{Trop}[x_1, \ldots, x_n]$. Then
\[
p \textrm{ is symmetric} \Leftrightarrow \operatorname{Sym}(p) = p.
\]
\end{proposition}
\begin{proof}
$(\Rightarrow)$ Suppose $p$ is symmetric. Then $p(x_{1}, x_2, \ldots, x_n) =  p(x_{\pi(1)}, x_{\pi(2)}, \ldots, x_{\pi(n)})$ for all $\pi\in S_n$ and $(x_1, x_2, \ldots, x_n) \in \RR^n$. Since $\oplus$ is idempotent,
\[
p(x_{1}, x_{2}, \ldots, x_{n}) = \bigoplus_{\pi\in S_n} p(x_{1}, x_{2}, \ldots, x_{n}) = \bigoplus_{\pi\in S_n} p(x_{\pi(1)}, x_{\pi(2)}, \ldots, x_{\pi(n)})
\]
for all $(x_1, x_2, \ldots, x_n) \in \RR^n$. Consequently, $\operatorname{Sym}(p) = p$.

$(\Leftarrow)$ By definition $\operatorname{Sym}(p)$ is symmetric and since $p = \operatorname{Sym}(p)$ so is $p$.
\end{proof}

The following symmetric tropical polynomials will play an important role in our discussion.
\begin{definition}
Given variables $x_1, \ldots, x_n$, we define the elementary symmetric tropical polynomials $e_1,\ldots, e_n \in \operatorname{Trop}[x_1, x_2, \ldots, x_n]$ by the formulas
\[
\begin{array}{lcl}
e_1 &=& x_1 \oplus \ldots \oplus x_n, \\
&\vdots& \\
e_k &=& \operatorname{Sym}( x_{1}  \odot \ldots \odot x_{k}),\\
&\vdots& \\
e_n &=& x_1\odot x_2 \odot \ldots \odot x_n.
\end{array}
\]
\end{definition}
The total degree of expression $e_k$ is $k$. Elementary symmetric tropical polynomials give coordinates on $\RR^n/ S_n$. In other words, they separate orbits.
\begin{proposition}
Let  $[(x_1, \ldots, x_n)]$ and $[(y_1, \ldots, y_n)]$ be two orbits under the $S_n$-action on $\RR^n$. If 
\[
e_i([(x_1, \ldots, x_n)]) = e_i([(y_1, \ldots, y_n)])
\]
for all $i$, then $[(x_1, \ldots, x_n)] = [(y_1, \ldots, y_n)]$.
\end{proposition}
\begin{proof}
Suppose $[(x_1, \ldots, x_n)]$ and $[(y_1, \ldots, y_n)]$ are orbits for which 
\[
e_i([(x_1, \ldots, x_n)]) = e_i([(y_1, \ldots, y_n)])
\]
 for all $i$. Let $(x_1, \ldots, x_n) \in [(x_1, \ldots, x_n)]$ be such that $x_1 \leq x_2 \leq \ldots \leq x_n$ and $(y_1, \ldots, y_n) \in [(y_1, \ldots, y_n)]$ such that ${y_1 \leq y_2 \leq \ldots \leq y_n}$. 
Since 
\[
e_1([(x_1, \ldots, x_n)]) = e_1([(y_1, \ldots, y_n)]),
\]
it follows that $x_{1} = y_{1} = e_1([(x_1, \ldots, x_n)])$. Next note that
\[
x_1+x_2 = e_2([(x_1, \ldots, x_n)])  = e_2([(y_1, \ldots, y_n)]) = y_1+y_2.
\]
Since $x_1=y_1$, this implies $x_2 = y_2$. We repeat these steps until we get $x_i = y_i$ for $i\leq n-1$. Lastly, 
\[
x_1 + x_2 + \ldots + x_n =e_n([(x_1, \ldots, x_n)]) = e_n([(y_1, \ldots, y_n)]) = y_1 + y_2 + \ldots + y_n.
\]
Since $x_i = y_i$ for $i\leq n-1$, it follows from this last equation that $x_n = y_n$ and we are done.
\end{proof}
The goal of the remainder of this section is to prove the following theorem, which states that elementary symmetric polynomials generate the symmetric tropical polynomials. 
\begin{theorem}\label{sym}
Every symmetric tropical polynomial in $\operatorname{Trop}[x_1, x_2, \ldots, x_n]$ can be written as a tropical polynomial in the elementary symmetric tropical polynomials $e_1, \ldots, e_n$ and $e_n^{-1}$. 
\end{theorem}
This theorem does not hold for polynomial semirings over general semirings. One counterexample, suggested by the editor, is the semiring of tropical polynomial expressions. Tropical expression $x^2\oplus y^2$ is not a polynomial in $e_1$,$e_2$ on the level of expressions, but equals $e_1^2$ in $\operatorname{Trop}[x,y]$. 

\begin{lemma}\label{refGstep1}
Let us suppose that a symmetric tropical polynomial $p$ is represented by 
\[
\bigoplus_{1\leq s \leq m}a_s \odot x_{1}^{i_1^s} \odot \ldots \odot x_{n}^{i_n^s}.
\]
Then
\[
p(x_1, x_2, \ldots, x_n) = \bigoplus_{1\leq s \leq m}  \operatorname{Sym}(a_s \odot x_{1}^{i_1^s} \odot \ldots \odot x_{n}^{i_n^s}).
\] 
\end{lemma}
\begin{proof}
Follows from Propositions~\ref{iffsym} and \ref{properties}.
\end{proof}

\begin{lemma}\label{Step1}
Suppose $i_{1}, i_{2}, \ldots, i_{k}$, $k\leq n$, are positive integers and ${a =\operatorname{min}(i_{1}, i_{2}, \ldots, i_{k})}$. Then
\[
e_{k}^a \odot \operatorname{Sym}(x_{1}^{i_{1}-a}  \odot \ldots \odot x_{k}^{i_{k}-a} ) = \operatorname{Sym}(x_{1}^{i_{1}}  \odot \ldots \odot x_{k}^{i_{k}} ).
\]
\end{lemma}
\begin{proof}[Proof of Lemma \ref{Step1}]
Since the Frobenius identity holds in tropical arithmetic, the expression on the left equals
\[
(\bigoplus_{\rho \in S_n} x_{\rho(1)}^a \odot \ldots \odot x_{\rho(k)}^a)  \odot (\bigoplus_{\pi\in S_n} x_{\pi(1)}^{i_{1}-a} \odot \ldots  \odot x_{\pi(k)}^{i_{k}-a}). 
\]
By distributivity and commutativity, we can rewrite it as 
\[
\bigoplus_{\rho \in S_n} \bigoplus_{\pi\in S_n} x_{\pi(1)}^{i_{1}-a} \odot \ldots  \odot x_{\pi(k)}^{i_{k}-a} \odot x_{\rho(1)}^a \odot  \ldots \odot x_{\rho(k)}^a.
\]
We must show that
\[
\bigoplus_{\rho \in S_n} \bigoplus_{\pi\in S_n} x_{\pi(1)}^{i_{1}-a} \odot \ldots  \odot x_{\pi(k)}^{i_{k}-a} \odot x_{\rho(1)}^a \odot  \ldots \odot x_{\rho(k)}^a = \bigoplus_{\sigma \in S_n} x_{\sigma(1)}^{i_{1}} \odot \ldots \odot x_{\sigma(k)}^{i_{k}}.
\]
The right hand side is bigger than the left hand side since the minimum is taken over a smaller set. We must show that
\[
x_{\pi(1)}^{i_{1}-a} \odot \ldots  \odot x_{\pi(k)}^{i_{k}-a} \odot x_{\rho(1)}^a \odot  \ldots \odot x_{\rho(k)}^a \geq \bigoplus_{\sigma \in S_n} x_{\sigma(1)}^{i_{1}} \odot \ldots \odot x_{\sigma(k)}^{i_{k}}
\]
for any $\pi, \rho\in S_n$ and the claim will follow. Let 
\[
M =\left \{ m \in \{1, 2, \ldots, k \}\,|\, j_m \in \{1, 2, \ldots, k \} \textrm{ exists such that }\pi(m) =\rho(j_m) \right \}
\]
and let 
\[
J =\left \{ j_m \in \{1, 2, \ldots, k \}\,|\, m \in \{1, 2, \ldots, k \} \textrm{ exists such that }\pi(m) =\rho(j_m) \right \}.
\]
We denote the elements of $M$ by $m_1, \ldots, m_l$, the elements of $\{1, 2, \ldots, k \}\setminus M$ by $s_1, \ldots, s_{k-l}$ and the elements of $\{1, 2, \ldots, k \}\setminus J$ by $q_1, \ldots, q_{k-l}$. We simplify the expression
\[{x_{\pi(1)}^{i_{1}-a} \odot \ldots  \odot x_{\pi(k)}^{i_{k}-a} \odot x_{\rho(1)}^a \odot  \ldots \odot x_{\rho(k)}^a} 
\]
to
\[
\bigodot_{r=1}^{l} x_{\pi(m_r)}^{i_{m_r}} \odot \bigodot_{r=1}^{k-l} x_{\pi(s_r)}^{i_{s_r}-a}  \odot \bigodot_{r=1}^{k-l} x_{\rho(q_r)}^{a}.
\]
For all $r =1, \ldots, k-l$
\[
\begin{array}{c}
x_{\pi(s_r)}^{i_{s_r}-a} \odot x_{\rho(q_r)}^{a} \geq  (x_{\pi(s_r)} \oplus x_{\rho(q_r)})^{i_{s_r}} = x_{\pi(s_r)}^{i_{s_r}} \oplus x_{\rho(q_r)}^{i_{s_r}}.\\
\end{array}
\]
Tropically multiplying (adding) these inequalities for applicable $r$ yields
\[
\bigodot_{r=1}^{l} x_{\pi(m_r)}^{i_{m_r}}  \odot \bigodot_{r=1}^{k-l} (x_{\pi(s_r)}^{i_{s_r}-a}  \odot  x_{\rho(q_r)}^{a}) \geq  \bigodot_{r=1}^{l} x_{\pi(m_r)}^{i_{m_r}}  \odot \bigodot_{r=1}^{k-l} (x_{\pi(s_r)}^{i_{s_r}} \oplus x_{\rho(q_r)}^{i_{s_r}}).
\]
By distributivity
\[
\bigodot_{r=1}^{l} x_{\pi(m_r)}^{i_{m_r}}  \odot \bigodot_{r=1}^{k-l} (x_{\pi(s_r)}^{i_{s_r}} \oplus x_{\rho(q_r)}^{i_{s_r}})
=
\bigoplus_{\sigma \in  P} x_{\sigma(1)}^{i_{m_1}} \odot \ldots \odot x_{\sigma(l)}^{i_{m_l}}  \odot x_{\sigma(l+1)}^{i_{s_1}} \odot \ldots \odot  x_{\sigma(k)}^{i_{s_{k-l}}},
\]
where $\sigma \in P$ if $\sigma(r) =  \pi(m_r)$ for $r \in \{1, \ldots, l\}$ and $\sigma(r) \in \{ \pi(s_{r-l}), \rho(q_{r-l}) \}$ for all $r \in \{l+1, \ldots, k\}$. Since $P \subseteq S_n$, 
\[
\bigoplus_{\sigma \in  P} x_{\sigma(1)}^{i_{m_1}} \odot \ldots \odot x_{\sigma(l)}^{i_{m_l}}  \odot x_{\sigma(l+1)}^{i_{s_1}} \odot \ldots \odot  x_{\sigma(k)}^{i_{s_{k-l}}}\geq  \bigoplus_{\sigma \in S_n} x_{\sigma(1)}^{i_{1}} \odot \ldots \odot x_{\sigma(k)}^{i_{k}},
 \]
 because we are taking the minimum over a bigger set on the right hand side.

\end{proof}

\begin{lemma}\label{Step2}
We can express the symmetrization of any tropical monomial with nonnegative powers as a tropical polynomial in the elementary symmetric polynomials $e_1, \ldots, e_n$.
\end{lemma}
\begin{proof}
We prove the statement by this induction on $\operatorname{Deg} p$ .

If $\operatorname{Deg} p = 0$, then $p \equiv a = a(e_1, \ldots, e_n)$.

Suppose now that we can express all symmetric tropical polynomials of the required form of total degree less than $m$ as tropical polynomials in $e_1, \ldots, e_n$.

Let
\[
p(x_1,\ldots, x_n) =\operatorname{Sym}(a \odot x_{1}^{i_1} \odot \ldots \odot x_{n}^{i_n}) = a \odot \operatorname{Sym}(x_{1}^{i_1} \odot \ldots \odot x_{n}^{i_n}),\]
where $\operatorname{Deg} p = \operatorname{deg}\operatorname{Sym}(a \odot x_{1}^{i_1} \odot \ldots \odot x_{n}^{i_n}) = m$ and $i_1, \ldots, i_n$ are nonnegative. 
 
Suppose exactly $i_{j_1}, i_{j_2}, \ldots, i_{j_k}$ are nonzero. By Lemma \ref{Step1} we can write
\[
p(x_1, \ldots, x_n) = a \odot  e_{k}^b \odot \operatorname{Sym}(x_{1}^{i_{j_1}-b}  \odot \ldots \odot x_{k}^{i_{j_k}-b} ) 
\]
where $b =  \operatorname{min}(i_{j_1}, \ldots, i_{j_k}) > 0$.

 Since 
\[
\operatorname{Deg}\operatorname{Sym}(x_{1}^{i_{j_1}-b}  \odot \ldots \odot x_{k}^{i_{j_k}-b}) < m,
\]
 the claim follows by induction for polynomials of the specified form.
\end{proof}

\begin{remark}
As remarked by the referee, Lemmas~\ref{Step1} and ~\ref{Step2} have a nice geometric
interpretation: the permutahedron is the Minkowski sum of its
hypersimplices.
\end{remark}

\begin{proof}[Proof of Theorem~\ref{sym}]
Let $p$ be a symmetric tropical polynomial in which all monomials have nonnegative powers.  By Lemma~\ref{refGstep1} we can write it as 
\[
p(x_1, \ldots, x_n) =  \bigoplus_{1\leq s \leq m}  \operatorname{Sym}(a_s \odot x_{1}^{i_1^s} \odot \ldots \odot x_{n}^{i_n^s}).
\]
By Lemma \ref{Step2} each $\operatorname{Sym}(a_s \odot x_{1}^{i_1^s} \odot \ldots \odot x_{n}^{i_n^s})$ can be written as a tropical polynomial in $e_1, \ldots, e_n$. Therefore so can $p$. 

Let $q$ be any symmetric tropical polynomial. We can write it as $q = \frac{q e_n^j}{e_n^j}$, where $j$ is such an integer that $qe_n^j$ is a symmetric tropical polynomial in which all monomials have nonnegative powers. 
\end{proof}
A symmetric tropical polynomial $p$ can be written in terms of elementary symmetric tropical polynomials in many ways. Therefore the uniqueness statement of the Fundamental Theorem of Symmetric Polynomials does not hold in the tropical setting. However, if we work with a particular tropical expression we can make an analogue claim. \emph{The minimal representation} of a tropical polynomial $p$ is such a tropical expression
\[
a_1\odot x_1^{i_1^1} x_2^{i_2^1} \ldots x_n^{i_n^1} \oplus a_2\odot x_1^{i_1^2} x_2^{i_2^2} \ldots x_n^{i_n^2}\oplus \ldots \oplus a_m\odot x_1^{i_1^m} x_2^{i_2^m} \ldots x_n^{i_n^m}
\]
 functionally equivalent to $p$ that for each $1\leq j \leq m$ there exists a point ${(x_1, x_2, \ldots, x_n) \in \RR^n}$, so that 
\[
a_j+i_1^jx_1 +\ldots +i_n^jx_n < \min_{1\leq s\leq m, s \neq j} (a_s+i_1^s x_1 +\ldots + i_n^s x_n).
\]

We define the \emph{minimal symmetrization} of a tropical polynomial $p$ as follows. The monomial $a_1\odot x_1^{i_1^1} x_2^{i_2^1} \ldots x_n^{i_n^1}$ appears in the minimal representation of $p$ if and only if $a_1\odot x_{\pi(1)}^{i_1^1} x_{\pi(2)}^{i_2^1} \ldots x_{\pi(n)}^{i_n^1}$ appears for any permutation $\pi \in S_n$. It might happen that some permutations yield the same monomial up to the order of factors --- in that case only one of these permuted monomials appears in the minimal representation. We call the collection of these non-repeating monomials the minimal symmetrization of $a_1\odot x_1^{i_1^1} x_2^{i_2^1} \ldots x_n^{i_n^1}$ and denote it by 
 \[
\operatorname{MinSym} (a_1\odot x_1^{i_1^1} x_2^{i_2^1} \ldots x_n^{i_n^1}).
 \]
$\operatorname{MinSym}p$ and $\operatorname{Sym}p$ are the same as functions.

\begin{corollary}[Uniqueness]\label{unique}
If we apply the algorithm used to prove Theorem~\ref{sym} to the minimal representation of a symmetric tropical polynomial, then the polynomial expression in $e_1, \ldots, e_n$ and $e_n^{-1}$ is also minimal.
\end{corollary}
\begin{proof}
Let $p$ be a symmetric tropical polynomial in $n$ variables. Its minimal representation is of the form
\[
\bigoplus_{j} \operatorname{MinSym}(a_j \odot  x_n^{i_n^j + \ldots + i_1^j-k}
 x_{n-1}^{i_n^j + \ldots + i_2^j-k} \odot \ldots  \odot x_2^{i_n^j + i_{n-1}^j-k} \odot x_1^{i_n^j-k}),
\]
where $i_1^j, \ldots, i_n^j$ are all positive integers.

Applying the algorithm to $p$, taking minimal symmetrizations at every step, produces 
\[
\bigoplus_{j} a_j \odot e_n^{i_n^j-k} \odot e_{n-1}^{i_{n-1}^j} \odot \ldots \odot e_1^{i_1^j}.
\]
Assume now that this expression in $e_1, \ldots, e_n$ is not minimal. This means that a $j_0$ exists and for any $(x_1, \ldots, x_n)$, an $l$ such that 
\[
 (a_{j_0} \odot e_n^{i_n^{j_0}-k} \odot e_{n-1}^{i_{n-1}^{j_0}} \odot \ldots \odot e_1^{i_1^{j_0}}) (x_1, \ldots, x_n) \geq  (a_{l} \odot e_n^{i_n^{l}-k} \odot e_{n-1}^{i_{n-1}^{l}} \odot \ldots \odot e_1^{i_1^{l}})(x_1, \ldots, x_n). 
\] 
Without loss of generality we may assume ${x_1 \leq x_2 \leq \ldots \leq x_n}$. This implies that a $j_0$ exists and for any $(x_1, \ldots, x_n)$ an $l$ so that 
\[
\begin{array}{lcl}
a_{l} \odot x_1^{i_n^{l}+ i_{n-1}^{l} + \ldots + i_1^{l}-k} \odot\ldots \odot x_n^{i_n^{l}-k} &=& a_{l} \odot (x_1\odot \ldots \odot x_n)^{i_n^{l}-k} \odot \ldots \odot x_1^{i_1^{l}}\\
&=& (a_{l} \odot e_n^{i_n^{l}-k} \odot e_{n-1}^{i_{n-1}^{l}} \odot \ldots \odot e_1^{i_1^{l}})(x_1, \ldots, x_n)\\
&\leq& (a_{j_0} \odot e_n^{i_n^{j_0}-k} \odot e_{n-1}^{i_{n-1}^{j_0}} \odot \ldots \odot e_1^{i_1^{j_0}}) (x_1, \ldots, x_n)\\
& = & a_{j_0} \odot (x_1\odot\ldots \odot x_n)^{i_n^{j_0}-k} \odot \ldots \odot x_1^{i_1^{j_0}} \\
&=& a_{j_0} \odot x_1^{i_n^{j_0}+ i_{n-1}^{j_0} + \ldots + i_1^{j_0}-k} \odot\ldots \odot x_n^{i_n^{j_0}-k}. \\
\end{array}
\] 
It follows that a term in the original expression must have been redundant. This is a contradiction.
\end{proof}

The following corollary is a tropical polynomial analogoue of the Fundamental Theorem of Symmetric Polynomials.
\begin{corollary}[Fundamental Theorem of Symmetric Tropical Polynomials]
Every symmetric tropical polynomial in $\operatorname{Trop}[x_1, x_2, \ldots, x_n]$ with monomials whose powers are nonnegative can be written as a tropical polynomial with monomials whose powers are nonnegative in the elementary symmetric tropical polynomials $e_1, \ldots, e_n$. If we apply the algorithm used to prove Theorem~\ref{sym} to the minimal representation of a tropical polynomial with monomials whose powers are nonnegative, then the expression in $e_1, \ldots, e_n$ is also minimal.
\end{corollary}

\section{Symmetric Rational Tropical Functions}

\begin{definition}
A rational tropical function $r \in \operatorname{RTrop}[x_1, x_2, \ldots, x_n]$ is symmetric if 
\[
r(x_1,\ldots, x_n) = r(x_{\pi(1)},\ldots ,x_{\pi(n)})
\]
for all permutations $\pi \in S_n$.
\end{definition}
We denote the algebra of symmetric rational tropical functions by $\operatorname{RTrop}[x_1, x_2, \ldots, x_n]^{S_n}$. We can extend $\operatorname{Sym}$ to $\operatorname{RTrop}[x_1, x_2, \ldots, x_n]$:
\[
\begin{array}{llcl}
\operatorname{Sym} \colon& \operatorname{RTrop}[x_1, \ldots, x_n]& \to &\operatorname{RTrop}[x_1, \ldots, x_n]^{S_n} \\
&&&\\
&  r(x_1, \ldots, x_n) &\mapsto& \bigoplus_{\pi\in S_n} r(x_{\pi(1)}, x_{\pi(2)}, \ldots, x_{\pi(n)}).
\end{array}
\]
The operator $\operatorname{Sym}$ is well-defined, additive, and commutes with tropical multiplication. A rational tropical function $r$  is symmetric if and only if $\operatorname{Sym}(r) = r$. 

\begin{theorem}\label{maxmintheorem}
Every symmetric rational tropical function function in $\operatorname{RTrop}[x_1, x_2, \ldots, x_n]$ can be written as a  rational tropical function in the elementary symmetric tropical polynomials $e_1, \ldots, e_n$.
\end{theorem}

\begin{proof}
Any rational tropical function $r$ may be written as
\[
r= p \odot q^{-1},
\]
where $p$ and $q$ are in $\operatorname{Trop}[x_1, x_2, \ldots, x_n]$ and whose monomials all have nonnegative powers. 

Let $(x_1, \ldots, x_n) \in \RR^n$. Since $r$ is symmetric, 
\[
p(x_{\pi(1)}, \ldots, x_{\pi(n)}) \odot q(x_1, \ldots, x_n) = p(x_1, \ldots, x_n) \odot q(x_{\pi(1)}, \ldots, x_{\pi(n)})
\]
for all $\pi\in S_n$. Tropically summing over $\pi\in S_n$ gives
\[
\oplus_{\pi\in S_n}  (p(x_{\pi(1)}, \ldots, x_{\pi(n)}) \odot q(x_1, \ldots, x_n))= \oplus_{\pi\in S_n} ( p(x_1, \ldots, x_n) \odot q(x_{\pi(1)}, \ldots, x_{\pi(n)})).
\]
By distributivity,
\[
(\oplus_{\pi\in S_n}p(x_{\pi(1)}, \ldots, x_{\pi(n)}) ) \odot q(x_1, \ldots, x_n)= p(x_1, \ldots, x_n) \odot  (\oplus_{\pi\in S_n}q(x_{\pi(1)}, \ldots, x_{\pi(n)})).
\]
Since this holds for all ${(x_1, \ldots, x_n) \in \RR^n}$,
\[
\operatorname{Sym}(p) \odot q= p \odot  \operatorname{Sym}(q),
\]
and consequently
\[
p\odot q^{-1} = \operatorname{Sym}(p) \odot \operatorname{Sym}(q)^{-1}.
\]
By Theorem \ref{sym} $\operatorname{Sym}(q)$ and $\operatorname{Sym}(p)$ are tropical polynomials in $e_1, \ldots, e_n$. Consequently $r$ is a rational tropical function in $e_1, \ldots, e_n$. 
\end{proof}

\section{$r$-Symmetric Tropical Polynomials}\label{multipoly}
A tropical polynomial in $n$ variables is symmetric if it is invariant under the action of $S_n$ that permutes the variables. We can generalize this definition as follows: a tropical polynomial in $nr$ variables, divided into n blocks of $r$ variables each, is $r$-symmetric if it is invariant under the action of $S_n$ that permutes the blocks while preserving the order of the variables within each block.

We state the relevant results for the case when $r=2$, but by induction we can prove similar statements for a general $r$ (with $r=2$ as the base case). We focus on $r=2$ because persistence barcodes, persistence analogoues of Betti numbers, are collections of intervals. Each interval is given as a point $(x, y)$ and represents a feature which is `born' at $x$ and which `dies' at $y$. Since the order of intervals does not matter, we must identify functions symmetric with respect to the action of $S_n$ on $(\RR^2)^n$ that permutes pairs. 

Fix $n$. Let the symmetric group $S_n$ act on the matrix of indeterminates
\[
X = \begin{pmatrix} 
x_{1, 1} & x_{1, 2} \\
x_{2, 1} & x_{2, 2} \\
\vdots & \vdots \\
x_{n, 1} & x_{n, 2} \\
\end{pmatrix}
\]
by left multiplication. We want to find a generating set for the subset of $\operatorname{Trop}[x_{1, 1}, x_{1, 2}, \ldots, x_{n, 2}]$ that is invariant under the action of $S_n$ described above.

\begin{definition}
A tropical polynomial $p \in \operatorname{Trop}[x_{1, 1}, x_{1, 2}, \ldots, x_{n, 2}]$  is  2-symmetric if 
\[
p(x_{1, 1}, x_{1, 2}, \ldots, x_{n, 1}, x_{n, 2}) = p(x_{\pi(1), 1}, x_{\pi(1), 2}, \ldots , x_{\pi(n), 1}, x_{\pi(n), 2})
\]
for every permutation $\pi \in S_n$.
\end{definition}

Given a tropical monomial in the variables $x_{1, 1}, x_{1, 2}, \ldots, x_{n, 2}$, we construct its \emph{exponent matrix} from the matrix $X$ by replacing each variable by its exponent. 

\begin{example}
Let $n=2$. The exponent matrix of $x_{1,1} \odot x_{2,2}$ is 
\[
\begin{pmatrix} 
1& 0 \\
0& 1\\
\end{pmatrix}.
\]
\end{example}

We define a \emph{symmetrization} map with respect to the row permutation action of $S_n$:
\[
\begin{array}{lcl}
 \operatorname{Trop}[x_{1, 1}, x_{1, 2}, \ldots, x_{n, 2}]& \to &\operatorname{Trop}[x_{1, 1}, x_{1, 2}, \ldots, x_{n, 2}]^{S_n} \\
&&\\
 p(x_{1, 1}, x_{1, 2},  \ldots, x_{n, 1}, x_{n, 2})& \mapsto & \bigoplus_{\pi\in S_n} p(x_{\pi(1), 1},  x_{\pi(1), 2}, \ldots, x_{\pi(n), 1}, x_{\pi(n), 2}).
\end{array}
\]
We denote this map by $\operatorname{Sym}_2$.

\begin{example}
Let $n=2$. The symmetrization of $x_{1,1} \odot x_{2,2}$ is 
\[
\operatorname{Sym}_2(x_{1,1} \odot x_{2,2}) = x_{1,1} \odot x_{2,2}  \oplus x_{2,1} \odot x_{1,2} .
\]
\end{example}

\begin{proposition}\label{multiprop}
Let $p(x_{1, 1}, \ldots, x_{n, 2}), q(x_{1, 1}, \ldots, x_{n, 2})\in \operatorname{Trop}[x_{1, 1},\ldots,  x_{n, 2}]$. Then:
\begin{enumerate}
\item
$\operatorname{Sym}_2(p\oplus q)(x_{1, 1}, \ldots, x_{n, 2}) = \operatorname{Sym}_2(p)(x_{1, 1},  \ldots, x_{n, 2}) \oplus \operatorname{Sym}_2(q)(x_{1, 1},  \ldots, x_{n, 2})$,
\item
$a \odot \operatorname{Sym}_2(p)(x_{1, 1}, \ldots, x_{n, 2})= \operatorname{Sym}_2(a \odot p)(x_{1, 1},  \ldots, x_{n, 2})$.
\item $p \textrm{ is 2-symmetric} \Leftrightarrow \operatorname{Sym}_2(p) = p.$
\end{enumerate}
\end{proposition}
\begin{proof}
This proof is similar to the proofs of Propositions \ref{properties} and \ref{iffsym}.
\end{proof}

We want to identify an equivalent of elementary symmetric tropical polynomials in this setting. Let 
\[
\mathscr{E}_n = \left\{ \begin{pmatrix} 
e_{1, 1} & e_{1, 2} \\
e_{2, 1} & e_{2, 2} \\
\vdots & \vdots \\
e_{n, 1} & e_{n, 2} \\
\end{pmatrix} \neq [0]_n^2 \,\mid \, e_{i, j} \in \{0, 1\} \textrm{ for }i=1,2,\ldots, n, 
\textrm{ and } j=1,2  \right\}.
\]
Each matrix $E\in \mathscr{E}_n$ determines a tropical monomial $P(E)$. 
\begin{example}
Let $n=3$. If
\[E=
\begin{pmatrix} 
1& 0 \\
1 & 0 \\
0 & 1 \\
\end{pmatrix},
\]
then $P(E) =x_{1, 1} \odot x_{2, 1} \odot x_{3, 2}$.
\end{example}

We denote  the set of orbits under the row permutation action on $\mathscr{E}_n$ by $\mathscr{E}_n/{S_n}$. Each orbit $\{E_1, E_2, \ldots E_m\}$ determines a 2-symmetric tropical polynomial 
\[P(E_1)\oplus P(E_2)\oplus \ldots \oplus P(E_m).
\]

\begin{definition}
We call 2-symmetric tropical polynomials that arise from orbits $\mathscr{E}_n/{S_n}$ elementary. We let $e_{(e_{1, 1}, e_{1, 2}), \ldots, (e_{n, 1}, e_{n, 2}) }$ denote the tropical polynomial that arises from the orbit
\[
\left[\begin{pmatrix} 
e_{1, 1} & e_{1, 2} \\
e_{2, 1} & e_{2, 2} \\
\vdots & \vdots \\
e_{n, 1} & e_{n, 2} \\
\end{pmatrix}\right].
\]
\end{definition}

\begin{example}
Let $n=2$. The set of orbits under the $S_2$ action is
\[
\mathscr{E}_2/S_2 = \left\{
\begin{array}{c}
\left[ \begin{pmatrix} 
1 & 1 \\
1 & 1 \\
\end{pmatrix}\right]\,,
\left[\begin{pmatrix} 
1 & 0 \\
1 & 1 \\
\end{pmatrix}\right],\,
\left[\begin{pmatrix} 
1 & 1 \\
0 & 1 \\
\end{pmatrix}\right],\,
\left[\begin{pmatrix} 
0 & 0 \\
1 & 1 \\
\end{pmatrix}\right],\\
\\

\left[\begin{pmatrix} 
1 & 0 \\
1 & 0 \\
\end{pmatrix}\right]
\left[\begin{pmatrix} 
1 & 0 \\
0 & 1 \\
\end{pmatrix}\right],
\left[\begin{pmatrix} 
0 & 1 \\
0 & 1 \\
\end{pmatrix}\right],\,
\left[\begin{pmatrix} 
0 & 1 \\
0 & 0 \\
\end{pmatrix}\right],\,
\left[\begin{pmatrix} 
1 & 0 \\
0 & 0 \\
\end{pmatrix}\right]
\end{array}
  \right\}.
\]
A few examples of elementary 2-symmetric tropical polynomials are:
\[\begin{array}{rcl}
e_{[(1,1)(1,1)]} &=&x_{1,1} \odot x_{1,2} \odot x_{2,2}  \odot x_{2,1}, 
\\
e_{[(1,0)(1,1)]} &=& x_{1,1} \odot x_{2,1} \odot x_{2,2} \oplus x_{1,1} \odot x_{1,2} \odot x_{2,1},
\\
e_{[(1,1)(0,1)]} &=& x_{1,1} \odot x_{1,2} \odot x_{2,2} \oplus x_{1,2} \odot x_{2,1} \odot x_{2,2},
\\
e_{[(1,0)(0,0)]} &=& x_{1,1} \oplus x_{2,1}.
\end{array}
\]
For simplicity we write $e_{[(1,0)]}$ instead of $e_{[(1,0)(0,0)]}$ when $n$ is clear from the context. Similarly $e_{[(1,1)^2]}$ represents $e_{[(1,1)(1,1)]}$.
\end{example}

Now we show that elementary 2-symmetric tropical polynomials give coordinates on $\RR^{2n}/ S_n$.
\begin{proposition}
Let  $[(x_{1}, y_{1}, \ldots, x_{n}, y_{n})]$ and $[(x_{1}', y_{1}', \ldots, x_{n}', y_{n}')]$ be two orbits under the row permutation action on $\RR^{2n}$. If 
\[
e([(x_{1}, y_{1}, \ldots, x_{n}, y_{n})]) = e([(x_{1}', y_{1}', \ldots, x_{n}', y_{n}')])
\]
for all elementary 2-symmetric tropical polynomials $e$, then 
\[
[(x_{1}, y_{1}, \ldots, x_{n}, y_{n})] = [(x_{1}', y_{1}', \ldots, x_{n}', y_{n}')].
\]
\end{proposition}
\begin{proof}
Suppose that $x_{1} \leq x_{2} \leq \ldots \leq x_{n}$, $x_{1}' \leq x_{2}' \leq \ldots \leq x_{n}'$, $y_{1} \leq y_{2} \leq \ldots \leq y_{n}$ and $y_{1}' \leq y_{2}' \leq \ldots \leq y_{n}'$.

Let  $[(x_{1}, y_{\pi(1)}, \ldots, x_{n}, y_{\pi(n)})]$ and $[ (x_{1}', y_{\rho(1)}', \ldots, x_{n}', y_{\rho(n)}')]$ be two orbits under the row permutation action on $\RR^{2n}$ that satisfy 
\[
e([(x_{1}, y_{\pi(1)}, \ldots, x_{n}, y_{\pi(n)})]) = e([ (x_{1}', y_{\rho(1)}', \ldots, x_{n}', y_{\rho(n)}')])
\]
for all elementary 2-symmetric tropical polynomials.

 Applying $e_{[(1,0)]}$, we get $x_{1}=x_{1}'$. Applying $e_{[(1,0), (1,0)]}$, we get
 \[
 x_{1}+x_{2} = x_{1}'+x_{2}'
 \]
 and from here $x_{2}=x_{2}'$ and so on. Finally, applying $e_{[ (1,0)^n ]}$ yields $x_{n}=x_{n}'$. 

We use a similar argument using $e_{[(0,1)]}, e_{[(0,1), (0,1)]}, \ldots, e_{[ (0,1)^n]}$ to show that 
\[
{y_{1} = y_1', y_{2} = y_2', \ldots, y_{n} = y_{n}'}.
\]
It remains to show that the orbits $[(x_{1}, y_{\pi(1)}, \ldots, x_{n}, y_{\pi(n)})]$ and $[ (x_{1}, y_{\rho(1)}, \ldots, x_{n}, y_{\rho(n)})]$ are the same. From the first orbit we choose such $(x_{1}, y_{\pi(1)}, \ldots, x_{n}, y_{\pi(n)})$ that if $x_i = x_{i+1}$ for some $i\in \{1, \ldots, n-1\}$, then $y_{\pi(i)} \leq y_{\pi(i+1)}$. We choose a representative from the second orbit $(x_{1}, y_{\rho(1)}, \ldots, x_{n}, y_{\rho(n)})$ the same way. We must show that $y_{\rho(i)} = y_{\pi(i)}$ for $i=1, \ldots, n$ and we do this by induction.

First we show that $y_{\pi(1)} = y_{\rho(1)}$. 

Let $I = \{k \in \{ 1, 2, \ldots, n \} \,\mid \, y_k < y_{\pi(1)} \}$. 
Using the fact that $x_1 \leq x_2 \leq \ldots \leq x_n$, we evaluate $e_{[ (1,1)(0, 1)^{|I|} ]}$ where $|I|$ is the cardinality of $I$:
\[
e_{[(1,1)(0,1)^{|I|}]}([(x_{1}, y_{\pi(1)}, \ldots, x_{n}, y_{\pi(n)})]) = x_1+ y_{\pi(1)} +  \sum_{k\in I}y_k.
\]
By definition
\[
e_{[(1,1)(0,1)^{|I|}]}([(x_{1}, y_{\rho(1)}, \ldots, x_{n}, y_{\rho(n)})]) = \min( \min_{ l \notin I } (x_l + \sum_{k\in I}y_k + y_{\rho(l)}),  \min_{ l \in I} (x_l + \sum_{k\in I}y_k + y_{\pi(1)})).
\]
If ${y_{\rho(1)} > y_{\pi(1)}}$, then
\[
\min_{ l \notin I } (x_l + \sum_{k\in I}y_k + y_{\rho(l)}) > x_1 + \sum_{k\in I}y_k +  y_{\pi(1)}
\]
and
\[
\min_{ l \in I } (x_l + \sum_{k\in I}y_k + y_{\pi(1)})) > x_1 + \sum_{k\in I}y_k +  y_{\pi(1)}
\]
since $y_{\rho(1)}$ is the minimum among the $y$-coordinates of pairs 
\[
\{ (x_j, y_{\rho(j)})\,\mid \, j\in \{1, 2, \ldots, n \}, x_j = x_1\}.
\]
Therefore ${y_{\rho(1)} > y_{\pi(1)}}$ contradicts the assumption that
\[
e([(x_{1}, y_{\pi(1)}, \ldots, x_{n}, y_{\pi(n)})]) = e([ (x_{1}', y_{\rho(1)}', \ldots, x_{n}', y_{\rho(n)}')])
\]
for all elementary 2-symmetric tropical polynomials. This proves that $y_{\rho(1)} \leq y_{\pi(1)}$.

A similar argument using $J = \{k \in \{ 1, 2, \ldots, n \} \,\mid \, y_k < y_{\rho(1)} \}$ and evaluating $e_{[ (1,1)(0,1)^{|J|} ]}$ shows that $y_{\pi(1)} \leq y_{\rho(1)}$.

We conclude that $y_{\pi(1)} = y_{\rho(1)}$. 

Now suppose that $y_{\pi(s)} = y_{\rho(s)}$ for all $s < m$, where $m$ is a positive integer. We want to show that $y_{\pi(m)} = y_{\rho(m)}$. Let us first suppose that $\pi(i) < \pi(m)$ for $i<m$. We set $I = \{k \in \{ 1, 2, \ldots, n \} \,\mid \, y_k < y_{\pi(m)} \}$. We evaluate at $e_{[ (1,1)^{m-1}(0, 1)^{|I|-m+2} ]}$ where $|I|$ is the cardinality of $I$ and conclude that $y_{\rho(m)} \leq y_{\pi(m)}$. 

Let $\{i_1, \ldots, i_{m-1}\}$ be a permutation of $\{\pi(1), \ldots, \pi(m-1)\}$ for which ${i_1 \leq i_2 \leq \ldots \leq i_{m-1}}$.

Now suppose that $i_{(m-2)} < \pi(m)<i_{(m-1)}$.
Let $I = \{k \in  \{ 1, 2, \ldots, n \}  \,\mid \, y_k < y_{\pi(m)} \}$. 
We evaluate at $e_{[(1,0) (1,1)^{m-2}(0, 1)^{|I|-m+3} ]}$ where $|I|$ is the cardinality of $I$ and conclude that then $y_{\rho(m)} \leq y_{\pi(m)}$. 

We use a similar argument in the following cases
\[
i_{(m-3)} < \pi(m)<i_{(m-2)}, i_{(m-4)} < \pi(m)<i_{(m-3)}, \ldots, i_{1} < \pi(m)<i_{2}
\]
 and conclude that ${y_{\rho(m)} \leq y_{\pi(m)}}$ in each.

Finally, let $\pi(m) < i_1$ and let $I = \{k \in  \{ 1, 2, \ldots, n \} \,\mid \, y_k < y_{\pi(m)} \}$. Applying $e_{[ (1,0)^{m-1}(1,1)(0, 1)^{|I|}]}$ we deduce that $y_{\rho(m)} \leq y_{\pi(m)}$.

So in all cases $y_{\rho(m)} \leq y_{\pi(m)}$. A similar argument holds for the permutation $\rho$. It follows that $y_{\pi(m)} = y_{\rho(m)}$. 
\end{proof}

Unfortunately an equivalent of Theorem~\ref{sym} does not hold for $r$-symmetric tropical polynomials. We show that when $r =2$ and $n = 2$ no finite set of generators exists for the 2-symmetric tropical polynomials.  

\begin{proposition}
$\operatorname{Trop}[x_{1, 1}, x_{1, 2},x_{2, 1}, x_{2, 2}]^{S_2}$ is not finitely generated.
\end{proposition}
\begin{proof}
Suppose a finite set of monomials, $\{ x_{1, 1}^{i_{1, 1}} \odot x_{1, 2}^{i_{1, 2}} \odot x_{2, 1}^{i_{2, 1}} \odot x_{2, 2}^{i_{2, 2}} \}_{i\in I}$, exists such that  
\[
 \operatorname{Sym}_2(I) = \{ \operatorname{Sym}_2(x_{1, 1}^{i_{1, 1}}\odot x_{1, 2}^{i_{1, 2}}\odot x_{2, 1}^{i_{2, 1}}\odot x_{2, 2}^{i_{2, 2}}) \}_{i\in I}
\]
 generates $\operatorname{Trop}[x_{1, 1}, x_{1, 2},x_{2, 1}, x_{2, 2}]^{S_2}$. Since $I$ is finite, an integer $d\leq 1$ exists such that
\[
0 \leq i_{1, 1}, i_{1, 2}, i_{2, 1}, i_{2, 2} \leq d-1
\]
for all $i\in I$.

We show by contradiction that $x_{1, 1}^{d} \odot x_{1, 2} \oplus x_{2, 1}^{d} \odot x_{2, 2}$ is not generated by 
\[
G = \{ \operatorname{Sym}_2(x_{1, 1}^{j_{1, 1}} \odot x_{1, 2}^{j_{1, 2}}\odot x_{2, 1}^{j_{2, 1}} \odot x_{2, 2}^{j_{2, 2}}) \}_{j_{1, 1}, j_{1, 2}, j_{2, 1}, j_{2, 2} \leq d-1},
\]
and consequently not by $\operatorname{Sym}_2(I)$. 

First note that this tropical polynomial expression representing the function is minimal. Suppose $x_{1, 1}^{d} \odot x_{1, 2} \oplus x_{2, 1}^{d} \odot x_{2, 2}$ can be generated by elements of $G$; that is, can be represented by a tropical expression $P$, which is a tropical sum of tropical products of the elements from $G$ (we can expand any expression to one of this form using distributivity). The term $x_{1, 1}^{d} \odot x_{1, 2}$ only appears in $P$ if $P$ contains a tropical product of symmetrizations of monomials $S_1 S_2 \ldots S_n$, $S_i \in G$, where for some $a$, ${0 \leq a < d}$, 
\[
S_1 = \operatorname{Sym}_2 (x_{1, 1}^{a} \odot x_{1, 2}) =x_{1, 1}^{a} \odot x_{1, 2} \oplus x_{2, 1}^{a} \odot x_{2, 2}, \textrm{ and }S_j = \operatorname{Sym}_2 (x_{1, 1}^{a_j}) =x_{1, 1}^{a_j} \oplus x_{2, 1}^{a_j}
\] 
for other $j$ such that $\sum_{j=2}^n a_j=d-a$ (this must hold because $P$ is functionally equivalent to $x_{1, 1}^{d} \odot x_{1, 2} \oplus x_{2, 1}^{d} \odot x_{2, 2}$). 

Expression $P$ contains $x_{2, 1}^{d} \odot x_{2, 2}$ by symmetry.  But it also contains, for example, $x_{1, 1}^{a} \odot x_{1, 2} \odot x_{2, 1}^{d-a}$. By assumption $P$ is functionally equivalent to $x_{1, 1}^{d} \odot x_{1, 2} \oplus x_{2, 1}^{d} \odot x_{2, 2}$, so $x_{1, 1}^{d} \odot x_{1, 2} \oplus x_{2, 1}^{d} \odot x_{2, 2}$ is the minimal expression for $P$ as argued earlier. This implies that
\[
x_{1, 1}^{a} \odot x_{1, 2} \odot x_{2, 1}^{d-a}\geq x_{1, 1}^{d} \odot x_{1, 2} \oplus x_{2, 1}^{d} \odot x_{2, 2}
\]
for all $x_{1, 1}, x_{1, 2}, x_{2, 1}, x_{2, 2} \in \RR$. This is equivalent to
\[
0\geq \min\{(d-a)(x_{1,1} - x_{2, 1}), a(x_{2,1}-x_{1,1}) + x_{2,2}-x_{1,2} \}
\]
at all points, which clearly does not hold. This contradicts our initial assumption and shows that the minimal form of $P$ cannot be $x_{1, 1}^{d} \odot x_{1, 2} \oplus x_{2, 1}^{d} \odot x_{2, 2}$ and in turn implies that $x_{1, 1}^{d} \odot x_{1, 2} \oplus x_{2, 1}^{d} \odot x_{2, 2}$ cannot be generated by $G$.
\end{proof}

\section{$r$-Symmetric Rational Tropical Functions}

\begin{definition}
A rational tropical function $r \in \operatorname{RTrop}[x_{1, 1}, x_{1, 2}, \ldots, x_{n, 1}, x_{n, 2}]$ is \linebreak {2-symmetric} if 
\[
r(x_{1, 1}, x_{1, 2}, \ldots, x_{n, 1}, x_{n, 2}) = r(x_{\pi(1), 1}, x_{\pi(1), 2}, \ldots , x_{\pi(n), 1}, x_{\pi(n), 2})
\]
for all permutations $\pi \in S_n$.
\end{definition}
We denote the algebra of 2-symmetric rational tropical functions by $\operatorname{RTrop}[x_{1, 1}, \ldots, x_{n, 2}]^{S_n}$. 
We can extend $\operatorname{Sym}_2$ to $\operatorname{RTrop}[x_{1, 1}, \ldots, x_{n, 2}]$.
\[
\begin{array}{llcl}
\operatorname{Sym}_2 \colon& \operatorname{RTrop}[x_{1, 1}, \ldots, x_{n, 2}]& \to &\operatorname{RTrop}[x_{1, 1}, \ldots, x_{n, 2}]^{S_n} \\
&&&\\
&  r(x_{1, 1}, x_{1, 2}, \ldots, x_{n, 1}, x_{n, 2}) &\mapsto& \bigoplus_{\pi\in S_n} r(x_{\pi(1), 1}, x_{\pi(1), 2}, \ldots, x_{\pi(n), 1}, x_{\pi(n), 2}).
\end{array}
\]
Operator $\operatorname{Sym}_2$ is well-defined, additive, and commutes with tropical multiplication. A  rational tropical function $r$ is 2-symmetric if and only if $\operatorname{Sym}_2(r) = r$.

\begin{theorem}\label{multimaxmintheorem}
Every symmetric rational tropical function in $\operatorname{RTrop}[x_{1, 1}, \ldots, x_{n, 2}]$ can be written as a  rational tropical function in the elementary {2-symmetric} tropical polynomials.
\end{theorem}

We prove this by induction with respect to a special order on tropical monomials with nonnegative powers.  Note that any such monomial ${x_{1, 1}^{j_{1, 1}} \odot x_{1, 2}^{j_{1, 2}}\odot \ldots \odot x_{n, 1}^{j_{n, 1}} \odot x_{n, 2}^{j_{n, 2}}}$ may be represented by a 2n-tuple $(j_{1,1}, j_{1, 2}, j_{2, 1}, j_{2,2}, \ldots, j_{n, 1}, j_{n, 2}) \in \ZZ_{\geq 0}^{2n}$. The number of nonzero entries in such a $2n$-tuple is a measure of how `spread out' a monomial is:
\[
S(x_{1, 1}^{j_{1, 1}} \odot x_{1, 2}^{j_{1, 2}}\odot \ldots \odot x_{n, 1}^{j_{n, 1}} \odot x_{n, 2}^{j_{n, 2}}) = |\left\{ (i, k) \in \{1, 2, \ldots, n\}\times \{1, 2\} \,|\, j_{i, k} \neq 0 \right\}|.
\]

Let 
\[
m_1 = x_{1, 1}^{j_{1, 1}} \odot x_{1, 2}^{j_{1, 2}}\odot \ldots \odot x_{n, 1}^{j_{ n, 1}} \odot x_{n, 2}^{j_{n, 2}}
\]
and
\[
m_2 = x_{1, 1}^{s_{1, 1}} \odot x_{1, 2}^{s_{1, 2}}\odot \ldots \odot x_{n, 1}^{s_{n, 1}} \odot x_{n, 2}^{s_{n, 2}}.
\]

Notice that $\textrm{Deg}(m_1) = \sum_{i=1}^n (j_{i, 1}+j_{i, 2} )$ and $\textrm{Deg}(m_2) = \sum_{i=1}^n (s_{i, 1}+s_{i, 2} )$. Then we define order $>_S$ as follows:
\[
\begin{array}{lcll}
m_1 >_S m_2 & \Leftrightarrow &  \operatorname{Deg}(m_1) > \operatorname{Deg}(m_2) \\
 & \textrm{or} & \operatorname{Deg}(m_1) = \operatorname{Deg}(m_2), S(m_1) < S(m_2) \\
 & \textrm{or} & \operatorname{Deg}(m_1) = \operatorname{Deg}(m_2), S(m_1) = S(m_2), m_1 >_{lex} m_2. \\
\end{array}
\]
\begin{example}
Let
\[
m_1 = x_{1, 1}\odot x_{1, 2} \odot x_{2, 1},
\]
and
\[
m_2 = x_{1, 1}\odot x_{1, 2}^2.
\]
We have $\operatorname{Deg}(m_1)= \operatorname{Deg}(m_2) = 3$, $S(m_1) =3$, $S(m_2)=2$. Since $S(m_2)<S(m_1)$, we have $m_2 >_S m_1$.
\end{example}

\begin{proposition}
Relation $>_S$ is a well-ordering on the set of tropical monomials with nonnegative powers.
\end{proposition}

\begin{lemma}\label{multiStep1}
The 2-symmetrization of any tropical monomial with nonnegative powers can be written as a rational tropical function in the elementary symmetric polynomials.
\end{lemma}
\begin{proof}
We prove the statement by induction on the set of tropical monomials with nonnegative powers well-ordered by $>_S$. Monomial ${x_{1, 1}^{0} \odot x_{2, 1}^{0}\odot \ldots \odot x_{1, n}^{0} \odot x_{2, n}^{0}} = 0$ and as such is a rational function of 2-symmetric elementary polynomials. By definition the symmetrization of a monomial of degree 1 is a 2-symmetric elementary polynomial.

Let us suppose that this statement holds for all monomials $x_{1, 1}^{s_{1, 1}}\odot \ldots \odot x_{n, 2}^{s_{n, 2}}$ with
\[
x_{1, 1}^{s_{1, 1}}\odot \ldots \odot x_{n, 2}^{s_{n, 2}} <_S x_{1, 1}^{i_{1, 1}}\odot \ldots \odot x_{n, 2}^{i_{n, 2}}.
\]
We must show that it holds for $x_{1, 1}^{i_{1, 1}}\odot \ldots \odot x_{n, 2}^{i_{n, 2}}$ and then the statement will follow by induction. If $i_{1, 1}, \ldots, i_{n, 2} \leq 1$, then $\operatorname{Sym}_2 (x_{1, 1}^{i_{1, 1}}\odot \ldots \odot x_{n, 2}^{i_{n, 2}})$ is an elementary symmetric polynomial by definition.

Otherwise, let us suppose that exactly $i_{j_1^1, j_2^1}, i_{j_1^2, j^2_2}, \ldots, i_{j_1^k, j_2^k}$ among $i_{1, 1}, \ldots, i_{n, 2}$ are positive. Let ${a =\operatorname{min}(i_{j_1^1, j_2^1}, i_{j_1^2, j^2_2}, \ldots, i_{j_1^k, j_2^k})}$ and 
\[
e = \operatorname{Sym}_2(x_{j_1^1, j_2^1}\odot \ldots \odot x_{j_1^k, j_2^k}).
\]
We observe the following expression:
\[
e^a \odot \operatorname{Sym}_2(x_{j_1^1, j_2^1}^{i_{j_1^1, j_2^1}-a}\odot \ldots \odot x_{j_1^k, j_2^k}^{i_{j_1^k, j_2^k}-a}).
\]
Since the Frobenius identity holds in tropical arithmetic, this expression equals
\[
(\bigoplus_{\rho \in S_n}x_{\rho(j_1^1), j_2^1}^a \odot \ldots \odot x_{\rho(j_1^k), j_2^k}^a) \odot (\bigoplus_{\pi\in S_n} x_{\pi(j_1^1), j_2^1}^{i_{j_1^1, j_2^1}-a}\odot \ldots \odot x_{\pi(j_1^k), j_2^k}^{i_{j_1^k, j_2^k}-a}).
\]
By distributivity, we can rewrite it as 
\[
\bigoplus_{\rho \in S_n} \bigoplus_{\pi\in S_n} x_{\rho(j_1^1), j_2^1}^a \odot \ldots \odot x_{\rho(j_1^k), j_2^k}^a \odot  x_{\pi(j_1^1), j_2^1}^{i_{j_1^1, j_2^1}-a}\odot \ldots \odot x_{\pi(j_1^k), j_2^k}^{i_{j_1^k, j_2^k}-a}.
\]
By commutativity and Frobenius identity this equals
\[
\operatorname{Sym}_2(x_{j_1^1, j_2^1}^{i_{j_1^1, j_2^1}}\odot \ldots \odot x_{j_1^k, j_2^k}^{i_{j_1^k, j_2^k}} )\oplus \bigoplus_{\rho \in S_n} \bigoplus_{\pi \in S_n^\rho} x_{\rho(j_1^1), j_2^1}^a \odot \ldots \odot x_{\rho(j_1^k), j_2^k}^a \odot  x_{\pi(j_1^1), j_2^1}^{i_{j_1^1, j_2^1}-a}\odot \ldots \odot x_{\pi(j_1^k), j_2^k}^{i_{j_1^k, j_2^k}-a}.
\]
Here $S_n^\rho = \{ \pi \in S_n \,\mid\, r\in\{1, \ldots, k\} \textrm{ exists such that } (\pi(j_1^r), j_2^r) \notin \{(\rho(j_1^s), j_2^s)\}_{s=1}^k \}$.

We denote $x_{\rho(j_1^1), j_2^1}^a \odot \ldots \odot x_{\rho(j_1^k), j_2^k}^a \odot  x_{\pi(j_1^1), j_2^1}^{i_{j_1^1, j_2^1}-a}\odot \ldots \odot x_{\pi(j_1^k), j_2^k}^{i_{j_1^1, j_2^1}-a}$ by $p_{\rho, \pi}$ and write: 
\[
\operatorname{Sym}_2(x_{j_1^1, j_2^1}^{i_{j_1^1, j_2^1}}\odot \ldots \odot x_{j_1^k, j_2^k}^{i_{j_1^k, j_2^k}} ) = e^a \odot \operatorname{Sym}_2(x_{j_1^1, j_2^1}^{i_{j_1^1, j_2^1}-a}\odot \ldots \odot x_{j_1^k, j_2^k}^{i_{j_1^k, j_2^k}-a}) \odot ( \bigoplus_{\rho \in S_n} \bigoplus_{\pi \in S_n^\rho} p_{\rho, \pi})^{-1}.
\]
Since $\bigoplus_{\rho \in S_n} \bigoplus_{\pi \in S_n^\rho} p_{\rho, \pi}$ is symmetric and $\operatorname{Sym}_2$ is additive,  
\[
\bigoplus_{\rho \in S_n}\bigoplus_{\pi\in S_n^\rho} p_{\rho, \pi} =\operatorname{Sym}_2 (\bigoplus_{\rho \in S_n}\bigoplus_{\pi\in S_n^\rho} p_{\rho, \pi})=  \bigoplus_{\rho \in S_n} \bigoplus_{\pi\in S_n^\rho} \operatorname{Sym}_2(p_{\rho, \pi}).
\]
Observe that
\[
p_{\rho, \pi}
 <_S x_{j_1^1, j_2^1}^{i_{j_1^1, j_2^1}}\odot \ldots \odot x_{j_1^k, j_2^k}^{i_{j_1^k, j_2^k}}
\]
for every $\pi \in S_n^\rho$ and $\rho\in S_n$. This holds since $\textrm{Deg}\, p_{\rho, \pi}= \textrm{Deg} \,x_{j_1^1, j_2^1}^{i_{j_1^1, j_2^1}}\odot \ldots \odot x_{j_1^k, j_2^k}^{i_{j_1^k, j_2^k}}$ and because by definition $S_n^\rho$ contains $\pi$ for which such $r$ exists that $(\pi(j_1^r), j_2^r) \notin \{(\rho(j_1^s), j_2^s)\}_{s=1}^k$ and therefore $S(p_{\rho, \pi}) \geq S(x_{j_1^1, j_2^1}^{i_{j_1^1, j_2^1}}\odot \ldots \odot x_{j_1^k, j_2^k}^{i_{j_1^k, j_2^k}}) +1$.

Moreover, 
\[
x_{j_1^1, j_2^1}^{i_{j_1^1, j_2^1}-a}\odot \ldots \odot x_{j_1^k, j_2^k}^{i_{j_1^k, j_2^k}-a} <_S x_{j_1^1, j_2^1}^{i_{j_1^1, j_2^1}}\odot \ldots \odot x_{j_1^k, j_2^k}^{i_{j_1^k, j_2^k}},
 \]
 so by the inductive hypothesis $ \operatorname{Sym}_2 (x_{j_1^1, j_2^1}^{i_{j_1^1, j_2^1}-a}\odot \ldots \odot x_{j_1^k, j_2^k}^{i_{j_1^k, j_2^k}-a})$ can be written as a  rational tropical function in 2-symmetric elementary tropical polynomials.

By induction $ \operatorname{Sym}_2 (p_{\rho, \pi})$ are  rational tropical functions in {2-symmetric} elementary tropical polynomials. 

Since we can express $ \operatorname{Sym}_2 (x_{j_1^1, j_2^1}^{i_{j_1^1, j_2^1}}\odot \ldots \odot x_{j_1^k, j_2^k}^{i_{j_1^k, j_2^k}})$ as a rational tropical function in 2-symmetric elementary tropical polynomials, the proof is complete. 
\end{proof}

 \begin{proof}[Proof of Theorem \ref{multimaxmintheorem}]
Any rational tropical function $r$ may be written as
\[
r(x_{1, 1}, \ldots, x_{n, 2}) = p(x_{1, 1}, \ldots, x_{n, 2}) \odot q(x_{1, 1}, \ldots, x_{n, 2})^{-1},
\]
where $p$ and $q$ are 2-symmetric tropical polynomials. It follows from Lemma \ref{multiStep1} that 2-symmetrization of any tropical polynomial with nonnegative powers can be written as a rational tropical function in the elementary symmetric polynomials (using additivity of $\operatorname{Sym}_2$). Consequently, $\operatorname{Sym}_2(p(x_{1, 1}, \ldots, x_{n, 2}))$ and $\operatorname{Sym}_2(q(x_{1, 1}, \ldots, x_{n, 2}))$ are  rational tropical functions in elementary 2-symmetric tropical polynomials as is their tropical quotient  $r(x_{1, 1}, \ldots, x_{n, 2})$. 
 \end{proof}
 %
\section{Discussion}
There are many other semirings of interest to tropical algebraists, such as the symmetrized $(\max, +)$ semiring~\cite{gaubertmaxplus}, the extended tropical semiring~\cite{extendedsemiring}, the supertropical semiring~\cite{Izhakian20102222}, etc. It would be interesting to see to see if the theorems represented in this paper have equivalents in those settings. 

{\large Acknowledgement}

The authors thank Gregory Brumfiel, Iurie Boreico and Davorin Le\v{s}nik for numerous discussions. We also thank the editor Chuck Weibel and the referees of this journal for careful reading of the manuscript and good suggestions. 

\printbibliography
 \end{document}